\documentclass[10pt]{amsart}
\usepackage[T1]{fontenc}
\usepackage[a4paper, margin=4cm, tmargin=4cm, bmargin=4cm]{geometry}
\usepackage[latin1] {inputenc}
\usepackage{amsmath}
\usepackage{amsfonts, amssymb, textcomp}

\usepackage{subeqnarray}

\usepackage{latexsym}
\usepackage{fancyhdr}
\usepackage{longtable}
\usepackage{amsmath, amssymb}
\usepackage{graphicx}
\usepackage{enumerate}

\usepackage{verbatim}
\usepackage{color}

\usepackage[hypertexnames=false]{hyperref}
\usepackage{cleveref}

\newcounter{mnotecount}[section]

\newcommand{\rmnote}[1]{}

\theoremstyle{plain}
\newtheorem{proposition}{Proposition}[section]
\newtheorem{theorem}[proposition]{Theorem}
\newtheorem{lemma}[proposition]{Lemma}

\theoremstyle{definition}

\newtheorem{example}[proposition]{Example}

\newtheorem{remark}[proposition]{Remark}

\newcommand{\RR}{\mathbb{R}}
\newcommand{\CC}{\mathbb{C}}
\newcommand{\NN}{\mathbb{N}}

\let\on=\operatorname

\title[Ultraholomorphic sectorial extensions of Beurling type]
{Ultraholomorphic sectorial extensions of Beurling type}

\author[D.N.~Nenning, A.~Rainer, and G.~Schindl]{David Nicolas Nenning, Armin Rainer, and Gerhard Schindl}

\address{Fakult\"at f\"ur Mathematik, Universit\"at Wien, Oskar-Morgenstern-Platz~1, A-1090 Wien, Austria.}
\email{david.nicolas.nenning@univie.ac.at}
\email{armin.rainer@univie.ac.at}
\email{gerhard.schindl@univie.ac.at}

\begin{document}

\begin{abstract}
We prove sectorial extension theorems for ultraholomorphic function classes of Beurling type defined by weight functions with a controlled loss of regularity. 
The proofs are based on a reduction lemma, due to the second author,
which allows to extract the Beurling from the Roumieu case, which was treated recently by Jim\'{e}nez-Garrido, Sanz, and the third author.
In order to have control on the opening of the sectors, where the extensions exist,
we use the (mixed) growth index
and the order of quasianalyticity of weight functions.
As a consequence we obtain corresponding extension results for classes defined by weight sequences.
Additionally, we give information on the existence of continuous linear extension operators.
\end{abstract}

\thanks{AR was supported by FWF-Project P 32905-N, DNN and GS by FWF-Project P 33417-N}
\date{\today}

\maketitle

\section{Introduction}\label{Introduction}

The aim of this work is to prove sectorial extension results of Borel--Ritt type.
For a given formal power series with admissible growth behavior of the coefficients one
looks for an ultraholomorphic function defined on a sector in the Riemann surface of the logarithm
and asymptotic to the given series. Ultraholomorphic functions are holomorphic functions which
satisfy certain growth conditions imposed on its iterated derivatives.
In this paper we are primarily interested in the case that the growth conditions are defined by a weight function;
in the spirit of Braun--Meise--Taylor classes \cite{BraunMeiseTaylor90}.
We allow for a controlled loss of regularity in the passage from formal power series to ultraholomorphic function.
This manifests itself by the use of two weight functions and their mixed growth index
which gives an upper bound on the opening of the sector on which the ultraholomorphic extension exists;
abbreviating, we call this the \emph{mixed setting}.
Specifically, we treat the mixed Beurling case (for precise definitions see \Cref{ultraspaces})
by reducing it to the mixed Roumieu case which was investigated by Jim\'{e}nez-Garrido, Sanz, and the third author in \cite{mixedsectorialextensions}.
This reduction procedure is based on a recent lemma proved and applied in a related context by the second author in \cite{beurlingultradiffmixedsetting}.

Let us briefly recall the historic background on ultraholomorphic sectorial extensions.
Classically, the problem was studied for Gevrey regularity, see Ramis \cite{zbMATH03636427}.
Thilliez \cite{Thilliezdivision} generalized the results to suitable weight sequences $M$
and associated with $M$ a \emph{growth index} $\gamma(M)$
which provides an upper bound for the opening of the sector on which the extension is defined.
The paper of Thilliez also extends earlier results of Schmets and Valdivia \cite{Schmetsvaldivia00}.
A different approach is pursued by Lastra, Malek, and Sanz \cite{LastraMalekSanzContinuousRightLaplace, Sanzsummability}.

In the recent papers \cite{sectorialextensions,sectorialextensions1} Jim\'{e}nez-Garrido, Sanz, and the third author
obtained analogous ultraholomorphic extension results for weight functions,
by transferring the ``complex'' method of \cite{LastraMalekSanzContinuousRightLaplace, Sanzsummability}
as well as the ``real'' method of \cite{Thilliezdivision}, respectively,
and by exploiting the technique of associating a \emph{weight matrix} with the given weight function, introduced in our article \cite{compositionpaper}.
In analogy to $\gamma(M)$, a growth index $\gamma(\omega)$ with similar properties was associated with a weight function $\omega$.

In all these works no loss of regularity occurs in the extension procedure.
The growth index $\gamma(\omega)$ is always dominated by the \emph{order of quasianalyticity} $\mu(\omega)$;
in general, one has the strict inequality $\gamma(\omega) < \mu(\omega)$, and the gap can be made arbitrarily large in individual examples.
While $\gamma(\omega)$ is connected to the existence of extensions
(on sectors of opening smaller than $\pi \gamma(\omega)$), the parameter $\mu(\omega)$ seems to be tied to the uniqueness of extensions
(on sectors of opening larger than $\pi \mu(\omega)$). (The latter statement is confirmed for certain $\omega$ which admit a weight sequence description
of the associated classes, but we conjecture that it holds in general.)

So, in order to have extensions on sectors of opening beyond $\pi \gamma(\omega)$, one is led
to allow for a controlled loss of regularity: one weight function $\sigma$ measures the regularity of the formal power series,
a second weight function $\omega$ that of its extension. The connection between $\sigma$ and $\omega$ is encoded in the
\emph{mixed growth index} $\gamma(\sigma,\omega)$, a natural generalization of $\gamma(\omega)$. In fact,
extensions exist on all sectors of opening smaller than $\pi \gamma(\sigma,\omega)$. Since $\gamma(\omega) \le \gamma(\sigma,\omega) \le \mu(\omega)$, with generally
strict inequalities, this means an improvement on the size of the sectors where extensions exist.
Moreover, extensions on sectors of all openings smaller than $\pi \mu(\omega)$ exist if $\sigma$ is allowed to depend on the opening.
These results were obtained in the Roumieu case by Jim\'{e}nez-Garrido, Sanz, and the third author
\cite{mixedsectorialextensions}; analogous statements hold for weight sequences.
It should be noted that, for technical reasons, the results involve a uniform shift of all weights
(i.e., a multiplication by the sequence $(p!)$ on the level of weight matrices)
which here we ignored for simplicity.
For a detailed study and comparison of the mentioned parameters we refer to \cite{index}.

The approach in \cite{mixedsectorialextensions} was the ``complex'' one of \cite{sectorialextensions},
since the ``real'' techniques of \cite{sectorialextensions1} failed in a crucial step.
At the time of writing \cite{mixedsectorialextensions} the mixed Beurling case could not be handled.
This changed thanks to a new reduction lemma proved by the second author in \cite{beurlingultradiffmixedsetting}
in order to deal with a similar situation concerning the ultradifferentiable Whitney extension problem.
Actually, this circle of ideas is intimately related to the problem at hand and was studied extensively in the literature.
We refer the interested reader to the (by no means exhaustive) list of papers
\cite{BonetBraunMeiseTaylorWhitneyextension}, \cite{BonetMeiseTaylorSurjectivity}, \cite{ChaumatChollet94}, \cite{Langenbruch94},
\cite{Schmetsvaldivia00}, \cite{surjectivity}, \cite{mixedramisurj},
 \cite{whitneyextensionmixedweightfunction}, \cite{whitneyextensionmixedweightfunctionII}, \cite{beurlingultradiffmixedsetting}.

In the setting of the present paper a weaker version of the aforementioned reduction lemma (namely, \Cref{lemma13generalizationultraholom})
suffices to \emph{fully} reduce the Beurling to the Roumieu case.
Thus it turns out that, again, the parameters $\gamma(\sigma,\omega)$ and $\mu(\omega)$
regulate the opening of the sectors on which extensions exist; see \Cref{Thm2ultraholombeurling}
and \Cref{Thm6ultraholombeurling}.
In addition, we provide sufficient conditions for the existence of continuous linear extension operators on suitable subspaces.

We point out that the reduction procedure in \cite{beurlingultradiffmixedsetting} involves a small loss of information, since it leads to a
stronger condition in the Beurling case.
Thanks to the ramified nature of the mixed strong non-quasianalyticity condition defining the index $\gamma(\sigma,\omega)$,
there is no loss of information in the ultraholomorphic sectorial extension problem.

The paper is organized as follows.
After discussing weight functions and sequences in \Cref{weightscond} and ultraholomorphic function and sequence spaces in \Cref{ultraspaces},
we prove in \Cref{mixedultraholomBeur} the main results on sectorial extension of mixed Beurling type for weight functions.
In \Cref{Thm2ultraholombeurling} the opening of the sector is controlled by $\gamma(\sigma,\omega)$ and in \Cref{Thm6ultraholombeurling} by $\mu(\omega)$. 
In the final \Cref{ultraholomweightsequ}, the results for weight functions are applied to the case that the growth conditions are defined in terms
of weight sequences, similarly using $\gamma(M,N)$ in \Cref{Thm4ultraholombeurling} and $\mu(N)$ in \Cref{Thm5ultraholombeurling}.
In all these theorems information on the existence of continuous linear extension operators is provided.

\section{Weights and conditions} \label{weightscond}
\subsection{Weight functions}
A function $\omega:[0,\infty)\rightarrow[0,\infty)$ is called \emph{weight function}
if it is continuous, non-decreasing, $\omega(0)=0$, and $\lim_{t\rightarrow\infty}\omega(t)=\infty$.
If in addition $\omega(t)=0$ for all $t\in[0,1]$, then $\omega$ is said to be \emph{normalized}.

Let us consider the following (standardly used) conditions:
\begin{align}
	\tag{$\omega_1$}\label{om1} &\text{$\omega(2t)=O(\omega(t))$ as $t\rightarrow\infty$.}
	\\
	\tag{$\omega_2$}\label{om2} &\text{$\omega(t)=O(t)$ as $t\rightarrow\infty$.}
	\\
	\tag{$\omega_3$}\label{om3} &\text{$\log(t)=o(\omega(t))$ as $t\rightarrow\infty$.}
	\\
	\tag{$\omega_4$}\label{om4} &\text{$\varphi_{\omega}:t\mapsto\omega(e^t)$ is a convex function on $\RR$.}
	\\
	\tag{$\omega_5$}\label{om5} &\text{$\omega(t)=o(t)$ as $t\rightarrow\infty$.}
	\\
	\tag{$\omega_6$}\label{om6} &\text{$\exists H\ge 1\;\forall t\ge 0:\;2\omega(t)\le\omega(H t)+H$.}
	\\
	\tag{$\omega_{\text{nq}}$}\label{omnq} &\int_1^{\infty}\frac{\omega(t)}{t^2}dt<\infty.
	\\
	\tag{$\omega_{\text{snq}}$}\label{omsnq} &\exists C>0\;\forall y>0 : \int_1^{\infty}\frac{\omega(y t)}{t^2}dt\le C\omega(y)+C.
\end{align}
Weight functions $\omega$ satisfying \eqref{omnq} are said to be \emph{non-quasianalytic} and 
those satisfying \eqref{omsnq} are called \emph{strongly non-quasianalytic} or simply \emph{strong}. 
Note that \eqref{omsnq} $\Rightarrow$ \eqref{omnq} $\Rightarrow$ \eqref{om5} $\Rightarrow$ \eqref{om2}.

For ease of reference we define the following sets of weight functions:
\begin{align*}
	\hypertarget{omset0}{\mathcal{W}_0} &:= \{\omega : \omega \text{ is a normalized weight function satisfying } 
	\hyperlink{om3}{(\omega_3)} \text{ and }\hyperlink{om4}{(\omega_4)}\},
	\\
	\hypertarget{omset1}{\mathcal{W}} &:= \{\omega\in\mathcal{W}_0 : \omega \text{ satisfies } \eqref{om1}\}.
\end{align*}
We remark that in \cite{beurlingultradiffmixedsetting} the elements of \hyperlink{omset1}{$\mathcal{W}$} (and only those) were called (normalized) weight functions.

For any $\omega\in\hyperlink{omset0}{\mathcal{W}_0}$ we define the \emph{Young conjugate} of $\varphi_{\omega}$ by
\begin{equation}\label{legendreconjugate}
\varphi^{*}_{\omega}(x):=\sup\{x y-\varphi_{\omega}(y): y\ge 0\}, \quad x\ge 0;
\end{equation}
it will appear in \Cref{omegaproperties} and in the definition of the ultraholomorphic classes in \Cref{ultraspaces}.

Given two weights $\sigma, \tau$ we write $\sigma\hypertarget{ompreceq}{\preceq}\tau$ if $\tau(t)=O(\sigma(t))$ as $t \to \infty$;
it reflects the inclusion relation of the corresponding ultraholomorphic classes, see \Cref{sec:inclusion}.
We call two weights $\sigma$ and $\tau$ \emph{equivalent} if $\sigma \preceq \tau$ and $\tau \preceq \sigma$.

\subsection{Weight functions obtained by power substitutions}

For any weight function $\omega$ and $r>0$ we denote by $\omega^r$ the weight $\omega^r(t):=\omega(t^r)$ resulting from the power substitution $t \mapsto t^r$.
Clearly, $(\omega^r)^s=\omega^{rs}$ for any $r,s>0$.

It is easy to see (cf.\ \cite[p.1635]{mixedsectorialextensions}) that $\omega^r \in \hyperlink{omset1}{\mathcal{W}}$ if and only if $\omega \in \hyperlink{omset1}{\mathcal{W}}$.
Furthermore, we have $\sigma\hyperlink{ompreceq}{\preceq}\tau$ if and only if $\sigma^r \hyperlink{ompreceq}{\preceq}\tau^r$.
In particular, $\sigma$ and $\tau$ are equivalent if and only if  $\sigma^r$ and $\tau^r$ are equivalent (for some/any $r>0$).

On the other hand, \eqref{omnq}, \eqref{omsnq},
and \eqref{om5} might, in general, not be preserved when passing from $\omega$ to $\omega^r$.
In fact (cf.\ \cite[(1)]{mixedsectorialextensions}), for $r>0$ the weight function $\omega^r$
is non-quasianalytic (i.e., satisfies \eqref{omnq}) if and only if $\omega$ fulfills
\begin{equation}
	\tag{$\omega_{\text{nq}_r}$} \label{omnqr}
	\int_1^{\infty}\frac{\omega(t)}{t^{1+1/r}}dt<\infty.		
\end{equation}

\subsection{Mixed growth index} \label{sec:MGI}

For weight functions $\omega,\sigma$ and $r>0$
we recall the condition (cf.\ \cite[Section 3.1]{mixedsectorialextensions} and also \cite[(5.4)]{mixedramisurj})
\begin{equation}
\tag*{$(\sigma,\omega)_{\gamma_r}$} \label{gammarfctmix}
\exists C>0\;\forall t\ge 0 :\int_1^{\infty}\frac{\omega(ty)}{y^{1+1/r}}dy\le C\sigma(t)+C.
\end{equation}
Note that, $\omega$ being non-decreasing, the integral is bounded below by $r \omega(t)$
so that this condition
implies $\sigma\hyperlink{ompreceq}{\preceq}\omega$.
Clearly, $(\sigma,\omega)_{\gamma_r}$ implies $(\sigma,\omega)_{\gamma_{r'}}$ for all $0<r'<r$.

The \emph{mixed growth index} is defined by
\begin{equation*}
\gamma(\sigma,\omega):=\sup\{r>0: \text{\ref{gammarfctmix} is satisfied}\}
\end{equation*}
and $\gamma(\sigma,\omega) := 0$ if \ref{gammarfctmix} holds for no $r>0$.
Putting $\gamma(\omega):=\gamma(\omega,\omega)$ we recover the growth index $\gamma(\omega)$ introduced and studied in \cite{index,sectorialextensions,sectorialextensions1}.

We have $\gamma(\omega)\le\gamma(\sigma,\omega)$ provided that $\sigma\hyperlink{ompreceq}{\preceq}\omega$, cf.\ \cite[Lemma 3]{mixedsectorialextensions}.
By \cite[Corollary 2.14]{index},  $\gamma(\omega)>0$ if and only if \eqref{om1} holds true.
In particular, for $\omega\in\hyperlink{omset1}{\mathcal{W}}$, $\gamma(\sigma,\omega)>0$ if and only if $\sigma\hyperlink{ompreceq}{\preceq}\omega$.

A weight function $\omega$ is strongly non-quasianalytic (i.e.\ satisfies \eqref{omsnq})
if and only if $\gamma(\omega)> 1$; see \cite[Corollary 2.13]{index}.
And, clearly, $\gamma(\sigma,\omega)>1$ implies that $\omega$ is non-quasianalytic and thus satisfies \eqref{om5}.

Note that $(\sigma,\omega)_{\gamma_r}$ if and only if $(\sigma^r,\omega^r)_{\gamma_1}$ and so
(cf.\ \cite[Remark 7 $(i)$]{mixedsectorialextensions})
    \begin{equation} \label{indextrans}
     \gamma(\sigma,\omega) =  r \gamma(\sigma^r,\omega^r) \quad \text{ for all } r>0.
    \end{equation}

\begin{remark}\label{mixedindexremark}
In \cite{beurlingultradiffmixedsetting} the condition $(\sigma,\omega)_{\gamma_{1/r}}$ was denoted by $(S_r)$ and the pair $(\omega,\sigma)$
was called $1/r$-strong.
\end{remark}

\subsection{Order of quasianalyticity} \label{lem:OQ}
The \emph{order of quasianalyticity} of a weight function $\omega$ is defined by (cf.\ \cite[$(18)$]{mixedsectorialextensions})
\begin{equation*}
\mu(\omega):=\sup\Big\{r>0: \int_1^{\infty}\frac{\omega(u)}{u^{1+1/r}}du<\infty\Big\}
=\sup\{r>0 : \omega \text{ satisfies } \eqref{omnqr}\}
\end{equation*}
and $\mu(\omega):=0$ if \eqref{omnqr} holds for no $r>0$.
It is preserved under equivalence of weight functions, since the condition \eqref{omnqr} is preserved.

We have $\gamma(\sigma,\omega)\le\mu(\omega)$ for any weight $\sigma\hyperlink{ompreceq}{\preceq}\omega$, cf.\ \cite[Lemma 7]{mixedsectorialextensions}.
Thus $\mu(\omega)>0$ if $\omega\in\hyperlink{omset1}{\mathcal{W}}$, by the properties of the mixed growth index, see \Cref{sec:MGI}.

\subsection{Weight sequences}

Any positive sequence $M=(M_p)\in\RR_{>0}^{\NN}$ is called \emph{weight sequence}. With $M$ we associate
the sequences $m=(m_p)$ and $\mu=(\mu_p)$ defined by $m_p:=\frac{M_p}{p!}$ and $\mu_p:=\frac{M_p}{M_{p-1}}$, $\mu_0:=1$, respectively.
A weight sequence $M$ is called \emph{normalized} if $1=M_0\le M_1$. 
For any weight sequence $M$ and $r>0$ we define the \emph{power} $M^{r}:=((M_p)^{r})_{p\in\NN}$.

A weight sequence $M$ is called \emph{log-convex} if
\begin{equation}
	\tag{$\text{lc}$} \label{lc}
	\forall p\in\NN_{>0} : M_p^2\le M_{p-1} M_{p+1},	
\end{equation}
which is equivalent to $\mu$ being non-decreasing. It is called \emph{strongly log-convex} if \eqref{lc} holds for the associated sequence $m$.
We say that $M$ has \emph{moderate growth} if
\begin{equation}
	\tag{$\text{mg}$}\label{mg}
	\exists C\ge 1 \;\forall p,q\in\NN : M_{p+q}\le C^{p+q} M_p M_q.	
\end{equation}
Replacing $M$ by $m$ or by $M^{r}$ (for arbitrary $r>0$) gives an equivalent condition. 
A weight sequence $M$ is called \emph{non-quasianalytic,} if
\begin{equation}
	\tag{$\text{nq}$} \label{mnq}
	\sum_{p=1}^{\infty}\frac{1}{\mu_p}<\infty.
\end{equation}
Note that $M^{1/r}$ is non-quasianalytic if and only if $M$ satisfies
\begin{equation}
	\tag{$\text{nq}_r$} \label{mnqr}
	\sum_{p=1}^{\infty}\Big(\frac{1}{\mu_p}\Big)^{1/r}<\infty.	
\end{equation}
For later reference we consider the conditions (cf.\ \cite{dissertation}, \cite{petzsche} and \cite{BonetMeiseMelikhov07})
\begin{align}
	\tag{$\gamma_1$} \label{gamma1}
	\sup_{p \in \NN_{>0}} &\frac{\mu_p}{p} \sum_{k \ge p} \frac{1}{\mu_k} < \infty,
	\\
	\tag{$\beta_3$} \label{beta3}
	\exists Q\in\NN_{>0} &: \liminf_{p\rightarrow\infty}\frac{\mu_{Qp}}{\mu_p}>1.
\end{align}
Two weight sequences $M$ and $N$ are said to be \emph{equivalent} if $C^{-1} \le \big(\frac{M_p}{N_p}\big)^{1/p} \le C$ for some $C>0$ (cf.\ \Cref{sec:inclusion}).

For ease of reference we introduce the set of weight sequences
$$\hypertarget{LCset}{\mathcal{LC}}:=\big\{M\in\RR_{>0}^{\NN}:\;M\;\text{is normalized, log-convex},\;\lim_{p \rightarrow\infty}(M_p)^{1/p}=\infty\big\}.$$
We shall recall the (mixed) growth index and the order of quasianalyticity for weight sequences in
\Cref{ultraholomweightsequ}.

\subsection{Associated function}

With $M\in\hyperlink{LCset}{\mathcal{LC}}$ one associates (cf.\ \cite[Chapitre I]{mandelbrojtbook} and \cite[Definition 3.1]{Komatsu73})
the function $\omega_M: [0,\infty) \to [0,\infty)$ defined by
\begin{equation*}\label{assofunc}
\omega_M(t):=\sup_{p\in\NN}\log\Big(\frac{t^p}{M_p}\Big)\;\;\;\text{for}\;t>0, \quad \omega_M(0):=0.
\end{equation*}
An easy calculation shows that, for all $r>0$,
\begin{equation} \label{omegaMspower}
\omega_M^r = r\omega_{M^{1/r}}.	
\end{equation}
We collect some well-known properties for $\omega_M$.

\begin{lemma}[{Cf.\ \cite[Lem. 2.4]{sectorialextensions} and \cite[Lem. 3.1]{sectorialextensions1}}]\label{assofuncproper}
Let $M\in\hyperlink{LCset}{\mathcal{LC}}$. Then:
\begin{itemize}
\item[(i)] $\omega_M$ belongs to \hyperlink{omset0}{$\mathcal{W}_0$}.

\item[(ii)] If $M$ satisfies $\eqref{gamma1}$, then $\omega_M$ fulfills \eqref{omsnq} (which in turn implies \eqref{om1}).

\item[(iii)] $M$ has moderate growth if and only if $\omega_M$ satisfies \eqref{om6}.
\end{itemize}
\end{lemma}

\subsection{Weight matrices}
Cf.\ \cite[Section 4]{compositionpaper}.
A \emph{weight matrix} $\mathcal{M}$ is a (one parameter) family of weight sequences $\mathcal{M}:=\{M^{[x]}: x\in \RR_{>0}\}$ such that
each $M^{[x]}$ is normalized and non-decreasing, and $M^{[x]} \le M^{[y]}$ if $x \le y$.
We call a weight matrix $\mathcal{M}$ \emph{standard log-convex}, abbreviated by \hypertarget{Msc}{$(\mathcal{M}_{\on{sc}})$},
if
$M^{[x]} \in\hyperlink{LCset}{\mathcal{LC}}$  for all $x >0$.

Weight matrices are a convenient technical tool for working with weight functions:

\begin{lemma}[{\cite[Section 5]{compositionpaper}}]\label{omegaproperties}
	With every $\omega\in\hyperlink{omset0}{\mathcal{W}_0}$ one can associate an \hyperlink{Msc}{$(\mathcal{M}_{\on{sc}})$} weight matrix
	$\Omega:=\{W^{[l]}: l>0\}$
	by setting
    $$W^{[l]}_j:=\exp\Big(\frac{1}{l}\varphi^{*}_{\omega}(lj)\Big).$$
    If $\omega$ additionally satisfies \eqref{om1}, then
     \begin{equation}\label{newexpabsorb}
     \forall h\ge 1 \;\exists A\ge 1\;\forall l>0\;\exists D\ge 1\;\forall j\in\NN :\ h^jW^{[l]}_j\le D W^{[Al]}_j.
     \end{equation}
    Moreover, $\omega \in \hyperlink{omset0}{\mathcal{W}_0}$ is non-quasianalytic if and only if some/each $W^{[l]}$ is non-quasianalytic.
  	All weight sequences $W^{[l]}$ are equivalent if and only if $\omega$ satisfies \eqref{om6}
  	which in turn is equivalent to some/each $W^{[l]}$ having moderate growth.
\end{lemma}

\section{Ultraholomorphic function classes and the Borel map} \label{ultraspaces}

We recall definitions and basic facts on ultraholomorphic classes; cf.\ \cite[Section 2.5]{sectorialextensions1}, \cite[Section 2.7]{mixedsectorialextensions},
\cite{Schmetsvaldivia00}, \cite{Thilliezdivision}, and references therein.

\subsection{Sectors}
Let $\mathcal{R}$ be the Riemann surface of the logarithm.
We wish to work in general unbounded open sectors in $\mathcal{R}$ with vertex at $0$, but all our results will be unchanged under rotation.
So it suffices to consider unbounded open sectors
$$S_{\gamma}:=\Big\{z\in\mathcal{R}: |\arg(z)|<\frac{\gamma\pi}{2}\Big\}, \quad \gamma >0,$$
of opening $\gamma\pi$
bisected by the positive real axis; we refer to them simply as \emph{sectors}.

\subsection{Ultraholomorphic classes associated with a weight sequence}
Let $M$ be a weight sequence, $S$ a sector, and $h>0$.
We consider the Banach space
$$\mathcal{A}_{M,h}(S):=\Big\{f\in\mathcal{H}(S): \sup_{z\in S, p\in\NN}\frac{|f^{(p)}(z)|}{h^p M_p}<\infty\Big\},$$
where $\mathcal{H}(S)$ is the space of holomorphic functions on $S$.
We define the spaces
\begin{equation*}
\mathcal{A}_{(M)}(S):=\bigcap_{h>0}\mathcal{A}_{M,h}(S) \quad \text{ and  } \quad \mathcal{A}_{\{M\}}(S):=\bigcup_{h>0}\mathcal{A}_{M,h}(S)
\end{equation*}
and equip them with their natural locally convex topologies. The Fr\'echet space $\mathcal{A}_{(M)}(S)$ is the \emph{ultraholomorphic class of Beurling type},
the (LB) space $\mathcal{A}_{\{M\}}(S)$ the \emph{ultraholomorphic class of Roumieu type} associated with $M$ in the sector $S$.

Analogously we introduce the sequence spaces
\begin{align*}
	\Lambda_{M,h} &:=\Big\{a=(a_p) \in\CC^{\NN}: \sup_{p\in\NN}\frac{|a_p|}{h^{p} M_{p}}<\infty\Big\},
	\\
	\Lambda_{(M)} &:=\bigcap_{h>0}\Lambda_{M,h},
	\quad \text{ and } \quad
	\Lambda_{\{M\}} :=\bigcup_{h>0}\Lambda_{M,h}.
\end{align*}
We have the (asymptotic) \emph{Borel maps} $\mathcal{B}:\mathcal{A}_{(M)}(S)\to \Lambda_{(M)}$ and $\mathcal{B}:\mathcal{A}_{\{M\}}(S)\to \Lambda_{\{M\}}$ given by
$f \mapsto(f^{(p)}(0))_{p\in\NN}$,
where $f^{(p)}(0):=\lim_{z\in S, z\rightarrow 0}f^{(p)}(z)$.

\subsection{Ultraholomorphic classes associated with a weight function}
Let $\omega$ be a normalized weight function satisfying \eqref{om3}. 
For a sector $S$ and $l>0$, we have the Banach space
$$\mathcal{A}_{\omega,l}(S) := \Big\{f\in\mathcal{H}(S) : \sup_{z\in S, p\in\NN}\frac{|f^{(p)}(z)|}{\exp(\frac{1}{l}\varphi^{*}_{\omega}(lp))}<\infty\Big\}.$$
We define the spaces
\begin{equation*}
\mathcal{A}_{(\omega)}(S):=\bigcap_{l>0}\mathcal{A}_{\omega,l}(S) \quad \text{ and  } \quad \mathcal{A}_{\{\omega\}}(S):=\bigcup_{l>0}\mathcal{A}_{\omega,l}(S)
\end{equation*}
and equip them with their natural locally convex topologies. The Fr\'echet space $\mathcal{A}_{(\omega)}(S)$ is the \emph{ultraholomorphic class of Beurling type},
the (LB) space $\mathcal{A}_{\{\omega\}}(S)$ the \emph{ultraholomorphic class of Roumieu type} associated with $\omega$ in the sector $S$.

Correspondingly, we have the sequence spaces
\begin{align*}
	\Lambda_{\omega,l} &:= \Big\{a=(a_p) \in\CC^{\NN}: \sup_{p\in\NN}\frac{|a_p|}{\exp(\frac{1}{l}\varphi^{*}_{\omega}(lp))}<\infty\Big\},
	\\
	\Lambda_{(\omega)} &:=\bigcap_{l>0}\Lambda_{\omega,l},
	\quad \text{ and } \quad
	\Lambda_{\{\omega\}} :=\bigcup_{l>0}\Lambda_{\omega,l}.	
\end{align*}
We get the \emph{Borel maps} $\mathcal{B}:\mathcal{A}_{(\omega)}(S)\to \Lambda_{(\omega)}$ and $\mathcal{B}:\mathcal{A}_{\{\omega\}}(S)\to \Lambda_{\{\omega\}}$.

For a weight matrix $\mathcal{M}=\{M^{[x]} : x>0\}$ and a sector $S$ we define ultraholomorphic classes of \emph{Beurling} and \emph{Roumieu type}
\begin{equation*}
\mathcal{A}_{(\mathcal{M})}(S):=\bigcap_{x>0}\mathcal{A}_{(M^{[x]})}(S), \quad \mathcal{A}_{\{\mathcal{M}\}}(S):=\bigcup_{x>0}\mathcal{A}_{\{M^{[x]}\}}(S),
\end{equation*}
as well as sequence spaces
\[	
	\Lambda_{(\mathcal{M})}:=\bigcap_{x>0}\Lambda_{(M^{[x]})}, \quad \Lambda_{\{\mathcal{M}\}}:=\bigcup_{x>0}\Lambda_{\{M^{[x]}\}},
\]
and equip them with their natural locally convex topologies. Clearly, we have the associated Borel maps.

\begin{proposition} \label{prop:topiso}
	Let $\omega\in\hyperlink{omset1}{\mathcal{W}}$ and let $\Omega$ be the associated weight matrix. Then 	
	\begin{equation*}
	\mathcal{A}_{(\omega)}(S)=\mathcal{A}_{(\Omega)}(S), \quad \mathcal{A}_{\{\omega\}}(S)=\mathcal{A}_{\{\Omega\}}(S),\quad
	\Lambda_{(\omega)}=\Lambda_{(\Omega)}, \quad \Lambda_{\{\omega\}}=\Lambda_{\{\Omega\}},
	\end{equation*}
	as locally convex vector spaces.
\end{proposition}

\begin{proof}
	This is a consequence of \cite[Lemma 5.9, (5.10)]{compositionpaper} (see \eqref{newexpabsorb}) and the way how the seminorms are defined in these spaces.
\end{proof}

\subsection{Inclusion relations} \label{sec:inclusion}
As an immediate consequence of the definitions we get the following inclusion relations (on any sector):
\begin{itemize}
	\item $\sup_{p\in\NN_{>0}}\big(\frac{M_p}{N_p}\big)^{1/p}<\infty$ 
	implies the inclusions $\mathcal{A}_{(M)}\subseteq\mathcal{A}_{(N)}$, $\mathcal{A}_{\{M\}}\subseteq\mathcal{A}_{\{N\}}$,
		$\Lambda_{(M)}\subseteq\Lambda_{(N)}$, and $\Lambda_{\{M\}}\subseteq\Lambda_{\{N\}}$.
	\item $\big(\frac{M_p}{N_p}\big)^{1/p} \to 0$, abbreviated by $M\hypertarget{mvartrian}{\vartriangleleft}N$,
	implies $\mathcal{A}_{\{M\}}\subseteq\mathcal{A}_{(N)}$ and $\Lambda_{\{M\}}\subseteq\Lambda_{(N)}$.
	\item $\tau(t)=O(\sigma(t))$ as $t \to \infty$ (i.e.\ $\sigma \preceq \tau$) implies $\mathcal{A}_{(\sigma)}\subseteq\mathcal{A}_{(\tau)}$, $\mathcal{A}_{\{\sigma\}}\subseteq\mathcal{A}_{\{\tau\}}$,
		$\Lambda_{(\sigma)}\subseteq\Lambda_{(\tau)}$, and $\Lambda_{\{\sigma\}}\subseteq\Lambda_{\{\tau\}}$.
	\item $\tau(t)=o(\sigma(t))$ as $t \to \infty$ implies	$\mathcal{A}_{\{\sigma\}}\subseteq\mathcal{A}_{(\tau)}$ and $\Lambda_{\{\sigma\}}\subseteq\Lambda_{(\tau)}$.
\end{itemize}
Obviously, we have $\mathcal{A}_{\ast}(S_r)\subseteq\mathcal{A}_{\ast}(S_{r'})$ for any $0<r'\le r$, where $*$ refers to any of the specified regularity classes.
All listed inclusions are continuous.

\section{Ultraholomorphic sectorial extensions}\label{mixedultraholomBeur}

In this section we prove the main \Cref{Thm2ultraholombeurling,Thm6ultraholombeurling}.
The proof of \Cref{Thm2ultraholombeurling} is based on \Cref{lemma13generalizationultraholom} which allows us to reduce the Beurling case to the
Roumieu case (treated in \cite{mixedsectorialextensions} and recalled in \Cref{Thm2ultraholomrecall,Cor1ultraholomrecall}).
\Cref{Thm6ultraholombeurling} is a corollary of \Cref{Thm2ultraholombeurling}.
We work with two weight functions $\omega$ and $\sigma$ allowing for a controlled loss of regularity in the extension procedure.
This generalizes the case $\omega=\sigma$ treated in \cite[Section 7]{sectorialextensions1}.

In the following, by an \emph{extension operator} we mean a continuous linear right-inverse of the Borel map
$\mathcal B$; the domain and codomain will be clear from the context.

\subsection{Notation for associated weights} \label{sec:notation}
\begin{itemize}
	\item For any weight function $\omega\in\hyperlink{omset0}{\mathcal{W}_0}$ we denote by $\mathcal{M}(\omega)$ the weight matrix $\{W^{[l]} : l >0\}$
	associated with $\omega$ in \Cref{omegaproperties}.
	\item For any weight matrix $\mathcal{M} = \{M^{[x]} : x>0\}$ consider the weight matrix $\widehat{\mathcal{M}} = \{\widehat{M}^{[x]} : x>0\}$,
	where $\widehat{M}^{[x]}_p := p! M^{[x]}_p$ for all $p \in \NN$.
	\item For a weight matrix $\mathcal{M}=\{M^{[x]} : x>0\}$ satisfying \hyperlink{Msc}{$(\mathcal{M}_{\on{sc}})$}
	 set $\omega(\widehat{\mathcal{M}}) := \omega_{\widehat{M}^{[1]}}$.
\end{itemize}
Then $\omega(\widehat{\mathcal{M}}) \in \hyperlink{omset1}{\mathcal{W}}$
and for all $\tau \in  \hyperlink{omset1}{\mathcal{W}}$ equivalent to $\omega(\widehat{\mathcal{M}})$ we have topological isomorphisms
$\mathcal{A}_{\{\tau\}}(S) \cong \mathcal{A}_{\{\widehat{\mathcal M}\}}(S)$ and
$\mathcal{A}_{(\tau)}(S) \cong \mathcal{A}_{(\widehat{\mathcal M})}(S)$
for all sectors $S$, analogously
$\Lambda_{\{\tau\}} \cong \Lambda_{\{\widehat{\mathcal M}\}}$ and
$\Lambda_{(\tau)} = \Lambda_{(\widehat{\mathcal M})}$;
see \cite[Theorem 5.3]{sectorialextensions} and \cite[Theorem 6.7]{sectorialextensions1}
(the proof is based on \cite[Corollary 3.17 $(ii)\Rightarrow(i)$]{testfunctioncharacterization}).
We point out that $\omega_{\widehat{M}^{[x]}}$ is equivalent to $\omega(\widehat{\mathcal{M}})$ for all $x>0$;
see \cite[Lemma 5.1, Cor. 5.2, Thm. 5.3]{sectorialextensions}.

Note that for $\omega(\mathcal{M}) := \omega_{M^{[1]}}$ condition \eqref{om1} and
the above properties might fail.

\subsection{Extensions of Beurling type controlled by the mixed growth index}
The main goal of this section is to prove the following theorem.
Recall that, for $\omega \in\hyperlink{omset1}{\mathcal{W}}$,
$\gamma(\sigma,\omega)>0$ if and only if $\sigma\hyperlink{ompreceq}{\preceq}\omega$ (see \Cref{sec:MGI}).

\begin{theorem}\label{Thm2ultraholombeurling}
Let $\omega,\sigma\in\hyperlink{omset1}{\mathcal{W}}$ and $0<\gamma<\gamma(\sigma,\omega)$.
Consider $\tau_1 := \omega(\widehat{\mathcal{M}(\sigma)})$ and $\tau_2 := \omega(\widehat{\mathcal{M}(\omega)})$.
Then:
\begin{itemize}
\item[(i)] We have the inclusion
$\mathcal{B}(\mathcal{A}_{(\tau_2)}(S_{\gamma}))\supseteq\Lambda_{(\tau_1)}$.
\item[(ii)]
If $\tau\in\hyperlink{omset1}{\mathcal{W}}$ satisfies $\sigma(t)=o(\tau(t))$ as $t\rightarrow\infty$,
then there exists an extension operator
$$\mathcal E^{\tau_3,\tau_2}: \Lambda_{\{\tau_3\}}\to \mathcal{A}_{(\tau_2)}(S_{\gamma}),$$
where $\tau_3 := \omega(\widehat{\mathcal{M}(\tau)})$.
\end{itemize}
\end{theorem}

The proof is based on a reduction to the Roumieu case which will be now recalled.

\subsection{Extensions of Roumieu type controlled by the mixed growth index}

\begin{theorem}[{\cite[Theorem 2]{mixedsectorialextensions}}]\label{Thm2ultraholomrecall}
Let $\sigma$ and $\omega$ be normalized weight functions satisfying \eqref{om3}.
Assume that $\gamma(\sigma,\omega)>0$ and let $0<\gamma<\gamma(\sigma,\omega)$.
Consider the weight matrices $\Sigma=\{S^{[x]}: x>0\} := \mathcal{M}(\sigma)$, $\Omega=\{W^{[x]}: x>0\} := \mathcal{M}(\omega)$, as well as $\widehat{\Sigma}$ and
$\widehat{\Omega}$.
Then there exists a constant $k_0>0$ such that for every $x>0$ and every $h>0$ we have an extension operator
$$
\mathcal E^{\sigma,\omega}_h: \Lambda_{\widehat{S}^{[x]},h}\to \mathcal{A}_{\widehat{W}^{[8x]},k_0h}(S_{\gamma}).
$$
Consequently, we have the inclusion
$\mathcal{B}(\mathcal{A}_{\{\widehat{\Omega}\}}(S_{\gamma}))\supseteq\Lambda_{\{\widehat{\Sigma}\}}$.
\end{theorem}

\begin{theorem}[{\cite[Corollary 1]{mixedsectorialextensions}}]\label{Cor1ultraholomrecall}
Let $\sigma, \omega\in\hyperlink{omset1}{\mathcal{W}}$ 
and $0<\gamma<\gamma(\sigma,\omega)$.
Consider $\tau_1 := \omega(\widehat{\mathcal{M}(\sigma)})$ and $\tau_2 := \omega(\widehat{\mathcal{M}(\omega)})$.
Then for every $l>0$ there exist $l_1>0$ and an extension operator
\begin{equation*} 
\mathcal E^{\tau_1,\tau_2}_{l} : \Lambda_{\tau_1,l}\to  \mathcal{A}_{\tau_2,l_1}(S_{\gamma}).
\end{equation*}
In particular, we have $\mathcal{B}(\mathcal{A}_{\{\tau_2\}}(S_{\gamma}))\supseteq\Lambda_{\{\tau_1\}}$.
\end{theorem}

\subsection{Reduction lemma}  \label{reductionlemma}

The reduction is based on the following variant of \cite[Lemma 13]{beurlingultradiffmixedsetting} 
(which in turn contains ideas from \cite[Lemma 4.4]{BonetBraunMeiseTaylorWhitneyextension}).
The proof simplifies significantly, because in contrast to \cite{beurlingultradiffmixedsetting}
we need not bother about concavity of the weights.

\begin{lemma} \label{lemma13generalizationultraholom}
Let $\omega,\sigma$ be (normalized) weight functions such that $\gamma(\sigma,\omega)>1$.
Let $f:[0,\infty)\rightarrow[0, \infty)$ satisfy $\sigma(t)=o(f(t))$ as $t\rightarrow\infty$.
Then there exist (normalized) weight functions $\widetilde{\omega},\widetilde{\sigma}$ satisfying $\gamma(\widetilde{\sigma},\widetilde{\omega})>1$ and
\begin{equation}\label{lemma13generalizationultraholomequ}
\omega(t)=o(\widetilde{\omega}(t)), \quad \sigma(t)=o(\widetilde{\sigma}(t)), \quad\widetilde{\sigma}(t)=o(f(t)) \qquad\text{ as }t\rightarrow\infty.
\end{equation}
If $\omega,\sigma\in\hyperlink{omset0}{\mathcal{W}_0}$ (resp.\ $\omega,\sigma\in\hyperlink{omset1}{\mathcal{W}}$),
then we may assume that also $\widetilde{\omega}, \widetilde{\sigma}\in\hyperlink{omset0}{\mathcal{W}_0}$ (resp.\ $\widetilde{\omega}, \widetilde{\sigma}\in\hyperlink{omset1}{\mathcal{W}}$).
\end{lemma}

\begin{proof}
By \cite[Proposition 7]{beurlingultradiffmixedsetting}, the condition $\gamma(\sigma,\omega)>1$ is equivalent to
\begin{equation}\label{equ41new}
\exists C>0\;\exists K>H>1\;\exists t_0\ge 0\;\forall t\ge t_0\;\forall j\in\NN_{>0} : \omega(K^{j}t)\le CH^{j}\sigma(t).
\end{equation}
We will construct weight functions $\widetilde{\omega}$ and $\widetilde{\sigma}$ satisfying \eqref{lemma13generalizationultraholomequ} and
\eqref{equ41new}, i.e., $\gamma(\widetilde{\sigma},\widetilde{\omega})>1$.

Note that $f(t)\rightarrow\infty$ as $t \to \infty$,
since $\sigma(t)\rightarrow\infty$ and $\sigma(t)=o(f(t))$  as $t \to \infty$.
We consider a strictly increasing sequence $(x_n)_{n\ge 1}$ tending to infinity, with $x_1:=0$, $x_2\ge 1$, and satisfying the following requirements for all $n\ge 2$:
\begin{align}{}
	\label{equ42new}
	 x_n &> \max\{2,K\} x_{n-1}+n,
	 \\
	 \label{equ43new}
	 f(t) &\ge n^2\sigma(t), \quad \text{ for all } t\ge x_n,
	 \\
	 \label{equ44new}
	 \omega(x_n) &\ge 2^{n-i}\omega(x_i), \quad \text{ for all } 1 \le i \le n-1,
	 \\
	 \label{equ45new}
	 \sigma(x_n) &\ge 2^{n-i}\sigma(x_i), \quad \text{ for all } 1 \le i \le n-1.
\end{align}
Then we define the weights $\widetilde{\omega}$ and $\widetilde{\sigma}$ as follows: for $n \ge 1$ and $t\in[x_n,x_{n+1})$ set
\begin{align*}
	\widetilde{\omega}(t) := n\omega(t)-\sum_{i=1}^{n}\omega(x_i) \quad \text{ and } \quad	
	\widetilde{\sigma}(t) := n\sigma(t)-\sum_{i=1}^{n}\sigma(x_i).
\end{align*}
By definition and since $x_2\ge 1$, $\widetilde{\omega}$ is normalized if $\omega$ is normalized;  analogously for $\widetilde{\sigma}$.
Moreover, both $\widetilde{\omega}$ and $\widetilde{\sigma}$ are non-decreasing, continuous, tending to infinity as $t\rightarrow\infty$, and vanish at $0$.
Note that $\widetilde{\omega}$ satisfies \eqref{om4} provided that $\omega$ does; similarly for $\widetilde{\sigma}$.	

As in \cite[Lemma 13]{beurlingultradiffmixedsetting} one shows that, for all $n \ge 2$ and all $t\in[x_n,x_{n+1})$,
\begin{align}\label{equ4849combi}
	(n-2)\omega(t)\le\widetilde{\omega}(t) &\le n\omega(t),
	\\ \label{equ410411combi}
	(n-2)\sigma(t)\le\widetilde{\sigma}(t) &\le n\sigma(t).
\end{align}
Consequently, $\omega(t)=o(\widetilde{\omega}(t))$ and $\sigma(t)=o(\widetilde{\sigma}(t))$ as $t\rightarrow\infty$.
In particular, $\widetilde{\omega}$ satisfies \eqref{om3} provided that $\omega$ does; similarly for $\widetilde{\sigma}$.
Hence $\widetilde{\omega},\widetilde{\sigma}\in\hyperlink{omset0}{\mathcal{W}_0}$ provided that $\omega,\sigma\in\hyperlink{omset0}{\mathcal{W}_0}$.
Combining \eqref{equ43new} and \eqref{equ410411combi} yields $\widetilde{\sigma}(t)=o(f(t))$ as $t\rightarrow\infty$.
So also \eqref{lemma13generalizationultraholomequ} is shown.

Now $\gamma(\widetilde{\omega},\widetilde{\sigma})>1$ follows from \eqref{equ41new}, \eqref{equ42new}, \eqref{equ4849combi}, and \eqref{equ410411combi}
as in the proof of \cite[Lemma 13]{beurlingultradiffmixedsetting}.
By a similar argument (see \emph{loc.\ cit.}),
$\widetilde \omega$, $\widetilde \sigma$ satisfy \eqref{om1} if $\omega$, $\sigma$ do so.
\end{proof}

We also need the following observation.

\begin{lemma} \label{lem:hatflip}
	Let $\omega,\sigma \in\hyperlink{omset1}{\mathcal{W}}$ satisfy $\omega(t) = o(\sigma(t))$ as $t \to \infty$,
	and consider  $\Omega = \{W^{[x]} : x>0\} := \mathcal M(\omega)$ and $\Sigma = \{S^{[x]} : x>0\} := \mathcal M(\sigma)$.
	Then
	\[
		\forall H>0 \; \forall x>0 \; \exists C>0 : \widehat S^{[x]} \le C\, \widehat W^{[Hx]}.
	\]
\end{lemma}

\begin{proof}
	\cite[Lemma 5.16]{compositionpaper} implies
	\[
		\forall H>0 \; \forall x>0 \; \exists C>0 : S^{[x]} \le C\, W^{[Hx]}.
	\]
	which is obviously equivalent to the assertion.
\end{proof}

\subsection{Proof of \texorpdfstring{\Cref{Thm2ultraholombeurling}}{Theorem 4.1}}

(i) The argument follows a well-known scheme used, e.g., in
the proofs of \cite[Theorem 4.5]{BonetBraunMeiseTaylorWhitneyextension}, \cite[Theorem 7.2]{sectorialextensions1}, and \cite[Theorem 2]{beurlingultradiffmixedsetting}.

Fix $r>0$ such that $\gamma<r<\gamma(\sigma,\omega)$.
We consider the weight functions $\omega^r,\sigma^r\in\hyperlink{omset1}{\mathcal{W}}$ which
satisfy $\gamma(\sigma^r,\omega^r) > 1$, by \eqref{indextrans}.

Set $\Sigma := \mathcal{M}(\sigma)$ and $\Omega := \mathcal{M}(\omega)$.
Let $\widehat a=(\widehat a_p) \in\Lambda_{(\tau_1)}=\Lambda_{(\widehat{\Sigma})}$ be given.
Our goal is to show that $\widehat a\in \mathcal{B}(\mathcal{A}_{(\widehat{\Omega})}(S_{\gamma}))=\mathcal{B}(\mathcal{A}_{(\tau_2)}(S_{\gamma}))$.

To this end we consider $a:=(a_p)=(\widehat a_p/p!)$ and the function
$$
g(t):=\log\max\{1,|a_p|\}, \quad p\le t<p+1,\;p\in\NN.$$
Since $a \in\Lambda_{(\Sigma)}=\Lambda_{(\sigma)}$, for each integer $j \ge 1$ there exists $C_j>0$ such that
\[
	g(t) \le j \varphi^*_\sigma (t/j) + C_j, \quad \text{ for all } t \ge 0.
\]
Using \cite[Lemma 4.3]{BonetBraunMeiseTaylorWhitneyextension} (for $\psi_j := j \varphi^*_\sigma (t/j)$), we conclude that there is a
convex function $h : [0,\infty) \to [0,\infty)$ and a positive sequence $(D_j)_j$ such that
\[
	g \le h \le \inf_{j \ge 1} (j \varphi^*_\sigma (t/j) + D_j).
\]
(In this step we need \eqref{om1} for $\sigma$ in order to assume w.l.o.g.\ that $\sigma$, and hence $\varphi_{\sigma}$, is of class $\mathcal{C}^1$,
and $\varphi_{\sigma}'(t) \to \infty$ as $t \to \infty$. Thus $\varphi_{\sigma}^*$ is differentiable and $(\varphi_{\sigma}^*)' = (\varphi_{\sigma}')^{-1}$.
See \cite[p.210]{BraunMeiseTaylor90}, \cite[Lemma 15]{beurlingultradiffmixedsetting} and also \cite[Theorem 4.5]{BonetBraunMeiseTaylorWhitneyextension}.)

Then the Young conjugate $h^*$ of $h$ satisfies
\[
	h^*(t) \ge j \varphi_\sigma(t) - D_j, \quad \text{ for all } t \text{ and all }j,
\]
and thus
\[
	\sigma(t) = \varphi_\sigma(\log t)  \le \frac{1}{j} f(t) + \frac{D_j}{j},
\]
where $f(t):=h^{*}(\max\{0,\log(t)\})$. Hence, $\sigma(t)=o(f(t))$ as $t\rightarrow\infty$, and
putting $f^r(t):=f(t^r)$ we have $\sigma^r(t)=o(f^r(t))$.

Let us apply \Cref{lemma13generalizationultraholom} to $\sigma^r$, $\omega^r$, and $f^r$ (instead of $\sigma$, $\omega$, and $f$ in the lemma).
We obtain weights $\widetilde{\sigma},\widetilde{\omega}\in\hyperlink{omset1}{\mathcal{W}}$ satisfying
$\gamma(\widetilde{\sigma},\widetilde{\omega})>1$	
and
\begin{equation*}
\omega^r(t)=o(\widetilde{\omega}(t)), \quad \sigma^r(t)=o(\widetilde{\sigma}(t)), \quad\widetilde{\sigma}(t)=o(f^r(t)) \qquad\text{ as }t\rightarrow\infty.
\end{equation*}
Hence, in view of \eqref{indextrans}, we have weights
$\widetilde{\sigma}^{1/r},\widetilde{\omega}^{1/r} \in\hyperlink{omset1}{\mathcal{W}}$ such that
\begin{gather} \label{gammatilde}
	\gamma(\widetilde{\sigma}^{1/r},\widetilde{\omega}^{1/r})>r,
	\\\label{growthtilde}
	\omega(t)=o(\widetilde{\omega}^{1/r}(t)), \quad \sigma(t)=o(\widetilde{\sigma}^{1/r}(t)), \quad\widetilde{\sigma}^{1/r}(t)=o(f(t)) \qquad\text{ as }t\rightarrow\infty.
\end{gather}
In particular, there is a constant $B>0$ such that $\widetilde{\sigma}^{1/r} \le f + B$, whence for all $t\ge 0$
\[
	\varphi_{\widetilde{\sigma}^{1/r}}(t) = \widetilde{\sigma}^{1/r}(e^t) \le f(e^t) + B = h^*(t) +B
\]
and so (since $h$ is convex)
\[
	g \le h =h^{**} \le \varphi_{\widetilde{\sigma}^{1/r}}^* +B.
\]
By the definition of $g$, we find
$a \in\Lambda_{\{\widetilde{\sigma}^{1/r}\}} = \Lambda_{\{\widetilde{\Sigma}^{1/r}\}}$,
where $\widetilde{\Sigma}^{1/r} := \mathcal{M}(\widetilde{\sigma}^{1/r})$,
which is equivalent to
$$\widehat a \in \Lambda_{\{\widehat{\widetilde{\Sigma}^{1/r}}\}}.$$
By \eqref{gammatilde}, we can apply \Cref{Thm2ultraholomrecall} to $\widetilde{\omega}^{1/r}$ and $\widetilde{\sigma}^{1/r}$ (and $\gamma=r$)
and conclude
$$\widehat a \in\mathcal{B}(\mathcal{A}_{\{\widehat{\widetilde{\Omega}^{1/r}}\}}(S_r)),$$	
where $\widetilde{\Omega}^{1/r} := \mathcal{M}(\widetilde{\omega}^{1/r})$.
By \eqref{growthtilde} and \Cref{lem:hatflip},
$\mathcal{A}_{\{\widehat{\widetilde{\Omega}^{1/r}}\}}(S_r) \subseteq \mathcal{A}_{(\widehat \Omega)}(S_r)$
which gives the assertion
because $\gamma < r$.

(ii) Fix $r>0$ such that $\gamma<r<\gamma(\sigma,\omega)$. Then $\gamma(\sigma^r,\omega^r) > 1$ as above.
The assumption $\sigma(t)=o(\tau(t))$ gives $\sigma^r(t)=o(\tau^r(t))$ as $t \to \infty$.
Applying \Cref{lemma13generalizationultraholom} to $\sigma^r$, $\omega^r$, and $\tau^r$ (instead of $\sigma$, $\omega$, and $f$ in the lemma)
and repeating the steps that led to \eqref{gammatilde} and \eqref{growthtilde},
yields weight functions
$\widetilde{\omega}^{1/r},\widetilde{\sigma}^{1/r}\in\hyperlink{omset1}{\mathcal{W}}$ satisfying \eqref{gammatilde}
and
\begin{equation}\label{growthtilde2}
\omega(t)=o(\widetilde{\omega}^{1/r}(t)), \quad \sigma(t)=o(\widetilde{\sigma}^{1/r}(t)), \quad\widetilde{\sigma}^{1/r}(t)=o(\tau(t)) \qquad\text{ as }t\rightarrow\infty.
\end{equation}
By \Cref{Thm2ultraholomrecall},
there exists $k_0>0$ such that for all $x>0$ and $h>0$ we have an extension operator
\[
	\Lambda_{(\widehat{\widetilde S^{1/r}})^{[x]},h} \to \mathcal{A}_{(\widehat{\widetilde W^{1/r}})^{[8x]},k_0h}(S_{r}), 	
\]
where $\widehat{\widetilde \Sigma^{1/r}} = \{(\widehat{\widetilde S^{1/r}})^{[x]} : x>0\}$
and $\widehat{\widetilde \Omega^{1/r}} = \{(\widehat{\widetilde W^{1/r}})^{[x]} : x>0\}$.
By \eqref{growthtilde2}, \Cref{lem:hatflip}, and \Cref{sec:inclusion},
we have continuous inclusions
\[
	\mathcal{A}_{(\widehat{\widetilde W^{1/r}})^{[8x]},k_0h}(S_{r})
	\hookrightarrow \mathcal{A}_{\{\widehat{\widetilde \Omega^{1/r}}\}}(S_{r})
	\hookrightarrow \mathcal{A}_{(\widehat{\Omega})}(S_{r})
	= \mathcal{A}_{(\tau_2)}(S_{r}).
\]

Let $T := \mathcal{M}(\tau)$.
We have the continuous inclusions $\Lambda_{\{T\}} = \Lambda_{\{\tau\}} \subseteq \Lambda_{(\widetilde \sigma^{1/r})} = \Lambda_{(\widetilde{\Sigma}^{1/r})}$,
again by \eqref{growthtilde2}. (Note that, for the first equality, \eqref{om1} for $\tau$ is needed; cf.\ \eqref{newexpabsorb} and \Cref{prop:topiso}).
Furthermore,
the linear mappings $\Lambda_{\{T\}} \to \Lambda_{\{\widehat T\}}$
and $\Lambda_{(\widetilde{\Sigma}^{1/r})} \to \Lambda_{(\widehat{\widetilde{\Sigma}^{1/r}})}$
given by
$a = (a_p) \mapsto \widehat a= (p!\, a_p)$ are topological isomorphisms with inverse $(\widehat a_p) \mapsto (\widehat a_p/p!)$.
Hence, the inclusion
$\Lambda_{\{\tau_3\}} = \Lambda_{\{\widehat T\}}  \subseteq  \Lambda_{(\widehat{\widetilde{\Sigma}^{1/r}})}$
is continuous.
Then the composite
\[
	\Lambda_{\{\tau_3\}}
	\hookrightarrow \Lambda_{(\widehat{\widetilde{\Sigma}^{1/r}})}
	\hookrightarrow \Lambda_{(\widehat{\widetilde S^{1/r}})^{[x]},h}
	\rightarrow \mathcal{A}_{(\widehat{\widetilde W^{1/r}})^{[8x]},k_0h}(S_{r})
	\hookrightarrow \mathcal{A}_{(\tau_2)}(S_{r})
	\hookrightarrow \mathcal{A}_{(\tau_2)}(S_{\gamma})
\]
is the required extension operator.
The proof is complete.

\subsection{Extensions controlled by the order of quasianalyticity}\label{orderofquasiomega}

It is possible to have extensions on sectors of opening up to $\pi \mu(\omega)$, if one permits that $\sigma$ depends on the opening.
We shall see that this is a consequence of \Cref{Thm2ultraholombeurling}. For a Roumieu version see \cite[Theorem 6]{mixedsectorialextensions}.

\begin{theorem}\label{Thm6ultraholombeurling}
Let $\omega\in\hyperlink{omset1}{\mathcal{W}}$. Then:
\begin{itemize}
\item[(i)] For any $0<r<\mu(\omega)$ there exists $\sigma\in\hyperlink{omset1}{\mathcal{W}}$ such that for all $0<\gamma<r$ we have
    $$\mathcal{B}(\mathcal{A}_{(\tau_2)}(S_{\gamma}))\supseteq \Lambda_{(\tau_1)},$$
where $\tau_1 := \omega(\widehat{\mathcal{M}(\sigma)})$ and $\tau_2 := \omega(\widehat{\mathcal{M}(\omega)})$.

\item[(ii)] If $\tau\in\hyperlink{omset1}{\mathcal{W}}$ satisfies $\sigma(t)=o(\tau(t))$ as $t\rightarrow\infty$, then there exists
an extension operator
$$\mathcal E^{\tau_3,\tau_2}: \Lambda_{\{\tau_3\}}\longrightarrow\mathcal{A}_{(\tau_2)}(S_{\gamma}),$$
where $\tau_3 := \omega(\widehat{\mathcal{M}(\tau)})$.

\item[(iii)] The weight function $\sigma$ is minimal (up to equivalence) among all $\tau\in\hyperlink{omset1}{\mathcal{W}}$ satisfying $\tau\hyperlink{ompreceq}{\preceq}\omega$
and $(\tau,\omega)_{\gamma_r}$.
\end{itemize}
\end{theorem}

\begin{remark}
Of course, \emph{minimality} refers to the relation $\hyperlink{ompreceq}{\preceq}$ which induces a partial ordering
on the set of equivalence classes of weight functions. The corresponding function (or sequence) space is then maximal; cf.\ \Cref{sec:inclusion}.  	
\end{remark}

\begin{proof}[Proof of \Cref{Thm6ultraholombeurling}]
Note that $\mu(\omega) > 0$, by \Cref{lem:OQ}.
For $0<r<\mu(\omega)$ we consider the weight function $\kappa^{1/r}_{\omega^{r}}(t)=\kappa_{\omega^{r}}(t^{1/r})$, where
\begin{equation*}\label{gammamixedvsomegamfctequ}
\kappa_{\omega}(t):=\int_1^{\infty}\frac{\omega(ty)}{y^2}dy=t\int_t^{\infty}\frac{\omega(y)}{y^2}dy.
\end{equation*}
Then $(\kappa^{1/r}_{\omega^r},\omega)_{\gamma_r}$ is valid by definition.
The weight function $\kappa^{1/r}_{\omega^r}$
has all properties defining $\hyperlink{omset1}{\mathcal{W}}$ except normalization which however
can be achieved by switching to an equivalent weight, say $\sigma$, by redefining $\kappa^{1/r}_{\omega^r}$ near $0$;
see \cite[Remark 1.2 $(b)$]{BonetBraunMeiseTaylorWhitneyextension} and
\cite[Remark 3.2 $(b)$]{BonetMeiseTaylorSurjectivity}.
Then $\gamma(\sigma,\omega)\ge r>\gamma$ and so the statement follows from \Cref{Thm2ultraholombeurling}.
The minimality of $\sigma$ is immediate from its definition and the relation $(\sigma,\omega)_{\gamma_r}$. Cf.\ \cite[p.1650]{mixedsectorialextensions}.
\end{proof}

\section{Applications to the weight sequence setting} \label{ultraholomweightsequ}

In this section we apply the extension results for weight functions to
classes defined by weight sequences.

\subsection{Mixed growth index \texorpdfstring{$\gamma(M,N)$}{gamma(M,N)}}

Cf.\ \cite[Section 3.1]{mixedsectorialextensions} and references therein.
For a weight sequence $M$ and $r>0$ we consider the condition
\begin{equation}
		\tag{$\gamma_r$} \label{gammar}
		\sup_{p\in\NN_{>0}}\frac{(\mu_p)^{1/r}}{p}\sum_{k\ge p}\Big(\frac{1}{\mu_k}\Big)^{1/r}<\infty.
\end{equation}
It is immediate that $M$ satisfies \eqref{gammar} if and only if $M^{1/r}$ satisfies \eqref{gamma1}.

For weight sequences $M, N$ such that $\mu/\nu$ is bounded, consider the condition
\begin{equation*}
\tag*{$(M,N)_{\gamma_r}$} \label{gammarmix}
\sup_{p\in\NN_{>0}}\frac{(\mu_p)^{1/r}}{p}\sum_{k\ge p}\Big(\frac{1}{\nu_k}\Big)^{1/r}<\infty.
\end{equation*}
The \emph{mixed growth index} is defined by
\begin{equation*}
\gamma(M,N):=\sup\{r>0:  \text{ \ref{gammarmix} is satisfied}\}
\end{equation*}
and $\gamma(M,N):=0$ if \ref{gammarmix} holds for no $r>0$.
Note that $\gamma(M):=\gamma(M,M)$ is the growth index used in \cite{Thilliezdivision}; see also \cite{index,injsurj}.

\begin{remark}\label{beta3remark}
Let $M\in\hyperlink{LCset}{\mathcal{LC}}$ be given.

(i) $M$ satisfies \eqref{beta3} if and only if $\gamma(M)>0$; this follows from \cite[Theorem 3.11 (v)$\Leftrightarrow$(vii)]{index} applied to $\beta=0$.

(ii) We have $\gamma(\omega_M)\ge\gamma(M)$ and equality holds if $M$ has moderate growth; see \cite[Corollary 4.6]{index}.

(iii) $\omega_M$ satisfies \eqref{om1} if and only if $\gamma(\omega_M)>0$; see \cite[Corollary 2.14]{index}.
    So, if $M\in\hyperlink{LCset}{\mathcal{LC}}$ has moderate growth,
    then $\omega_M$ satisfies \eqref{om1} (i.e., $\omega_M \in\hyperlink{omset1}{\mathcal{W}}$) if and only if $M$ satisfies \eqref{beta3}.
    In general, for a sequence $N \in\hyperlink{LCset}{\mathcal{LC}}$ (not necessarily having moderate growth),
$\omega_N$ has the property \eqref{om1} if and only if
\[
	\exists L \in \NN_{>0} : \liminf_{p \to \infty} \frac{(N_{Lp})^{1/(Lp)}}{(N_p)^{1/p}} > 1,
\]
as it is shown in \cite[Theorem 3.1]{Schindl2021}.

(iv) These statements are consistent with the implication \cite[Lemma 12, $(2)\Rightarrow(4)$]{BonetMeiseMelikhov07}.
\end{remark}

\subsection{Extensions controlled by the mixed growth index}

The following theorem is a Beurling version of \cite[Theorem 4]{mixedsectorialextensions}.

\begin{theorem}\label{Thm4ultraholombeurling}
Let $M,N\in\hyperlink{LCset}{\mathcal{LC}}$ be such that
$\mu/\nu$ is bounded, $M$ has moderate growth, and
$\omega_M,\omega_N\in\hyperlink{omset1}{\mathcal{W}}$.
Then: 
\begin{itemize}
\item[(i)] $\gamma(M,N) = \gamma(\omega_M,\omega_N)>0$.

\item[(ii)] For any $0<\gamma<\gamma(M,N)$ we have the inclusion
$\mathcal{B}(\mathcal{A}_{(\widehat{N})}(S_{\gamma}))\supseteq\Lambda_{(\widehat{M})}$.

\item[(iii)] Let $L\in\hyperlink{LCset}{\mathcal{LC}}$ satisfy $L\hyperlink{mvartrian}{\vartriangleleft}M$ and assume that $\omega_L\in\hyperlink{omset1}{\mathcal{W}}$. Then there exists an extension operator
    $$\mathcal E^{L,N}:\Lambda_{\{\widehat{L}\}}\longrightarrow\mathcal{A}_{(\widehat{N})}(S_{\gamma}).$$
\end{itemize}
\end{theorem}

\begin{proof}
(i) By \cite[Lemma 4]{mixedsectorialextensions} we have $\gamma(M,N)=\gamma(\omega_M,\omega_N)$.
Condition \eqref{om1} for $\omega_N$ yields $\gamma(\omega_N)>0$ (see \Cref{beta3remark}) and so $\gamma(\omega_M,\omega_N)\ge\gamma(\omega_N)>0$ (by \Cref{sec:MGI}).

(ii) 
Let $\Omega := \mathcal{M}(\omega_N)$ and $\Sigma := \mathcal{M}(\omega_M)$.
Since $M$ has moderate growth, all sequences in $\Sigma =\{S^{[x]}: x>0\}$ are equivalent (see \Cref{assofuncproper,omegaproperties}), hence the same holds for
$\widehat{\Sigma}$.
The proof of \cite[Theorem 6.4]{testfunctioncharacterization} yields $S^{[1]}= M$, hence $\widehat{S}^{[1]}=\widehat{M}$ and so
$\Lambda_{(\widehat{M})}=\Lambda_{(\widehat{S}^{[1]})}=\Lambda_{(\widehat{\Sigma})}=\Lambda_{(\tau_1)}$
for the weight function $\tau_1 = \omega_{\widehat{M}} \in\hyperlink{omset1}{\mathcal{W}}$.
By \Cref{Thm2ultraholombeurling} applied to to $\omega_M$ and $\omega_N$, we conclude, for $\tau_2 = \omega(\widehat{\Omega})\in\hyperlink{omset1}{\mathcal{W}}$,
$$\Lambda_{(\widehat{M})} =\Lambda_{(\tau_1)} \subseteq \mathcal{B}(\mathcal{A}_{(\tau_2)}(S_{\gamma}))
=\mathcal{B}(\mathcal{A}_{(\widehat{\Omega})}(S_{\gamma}))\subseteq\mathcal{B}(\mathcal{A}_{(\widehat{N})}(S_{\gamma}));$$
the last inclusion is clear by the definition of the classes and since $\widehat{N}=\widehat{W}^{[1]} \in \widehat{\Omega}$.

(iii) The relation $L\hyperlink{mvartrian}{\vartriangleleft}M$ implies, by the definition of associated weight functions,
\[
\forall A\ge 1\;\exists C\ge 1\;\forall t\ge 0 : \omega_{M}(At)\le\omega_{L}(t)+C.
\]
In combination with the fact that $\omega_{M}$ satisfies \eqref{om6}, since $M$ has moderate growth (see \Cref{assofuncproper}),
we infer that
$\omega_{M}(t)=o(\omega_{L}(t))$ as $t\rightarrow\infty$.
Now it suffices to apply \Cref{Thm2ultraholombeurling}(ii) to $\omega_{L}$, $\omega_{M}$, and $\omega_{N}$ (instead of $\tau$, $\sigma$, and $\omega$)
and to note that $\mathcal{A}_{(\tau_2)}(S_{\gamma})\hookrightarrow \mathcal{A}_{(\widehat N)}(S_{\gamma})$
and
$\Lambda_{\{\widehat{L}\}} \hookrightarrow \Lambda_{\{\widehat{\mathcal{M}(\omega_L)}\}}$; see \Cref{sec:inclusion}.
\end{proof}

\begin{example}\label{Thm4ultraholombeurlingrem}
Here is an explicit example of sequences which fulfill the assumptions of \Cref{Thm4ultraholombeurling}
and underline its value:
Let $\gamma> \gamma' >1$.
By \cite[Lemma 13, Theorem 7]{mixedsectorialextensions}, there exist sequences $M,M'\in\hyperlink{LCset}{\mathcal{LC}}$
having moderate growth
such that
\begin{itemize}
	\item $p^\gamma \le \mu_p \le p^{\gamma(2\gamma-1)}$ and $p^{\gamma'} \le \mu'_p \le p^{\gamma'(2\gamma'-1)}$ for all $p\in \NN$,
	\item $\gamma(M) = \gamma(M') = 0$.
\end{itemize}
For $\varepsilon>0$ set $M_{\varepsilon} := (p!^{\varepsilon} M_p)$ and
$M_\varepsilon':=(p!^{\varepsilon} M'_p)$.
Then $\gamma(M_\varepsilon) = \gamma(M'_\varepsilon) = \varepsilon>0$ (see \cite[Theorem 3.11]{index})
and thus $\omega_{M_\varepsilon}$ and $\omega_{M'_\varepsilon}$ satisfy \eqref{om1} (see \Cref{beta3remark}).
By construction, $M_\varepsilon$ and $M'_\varepsilon$ have moderate growth.
If we additionally assume that $$\gamma'(2\gamma'-1)\le\gamma,$$
then $\mu_\varepsilon' \le \mu_\varepsilon$ and
$(M'_\varepsilon,M_\varepsilon)_{\gamma_r}$, for all $0<r<\gamma$, thus
$\gamma(M'_\varepsilon,M_\varepsilon)\ge\gamma$.
So \Cref{Thm4ultraholombeurling} can be applied to $M'_\varepsilon$ and $M_\varepsilon$. 
Note that, by choosing $\gamma$ and $\varepsilon$ appropriately,
one can make $\gamma(M'_\varepsilon,M_\varepsilon)
= \gamma(\omega_{M'_\varepsilon},\omega_{M_\varepsilon})$ arbitrarily large
and $\gamma(M_\varepsilon) = \gamma(M'_\varepsilon)
= \gamma(\omega_{M_\varepsilon}) = \gamma(\omega_{M'_\varepsilon})  >0$  arbitrarily small.
\end{example}

\subsection{Order of quasianalyticity \texorpdfstring{$\mu(N)$}{mu(N)}}

In analogy to \Cref{orderofquasiomega}
we consider the \emph{order of quasianalyticity} for a weight sequence
$N\in\hyperlink{LCset}{\mathcal{LC}}$ (see \cite[Section 3.2]{mixedsectorialextensions}):
\begin{equation*}
\mu(N):=\sup\Big\{r>0: \sum_{k\ge 1}\Big(\frac{1}{\nu_k}\Big)^{1/r}<\infty\Big\}
=\sup\{r>0: N \text{ satisfies } \eqref{mnqr}\}
\end{equation*}
and $\mu(N) := 0$ if \eqref{mnqr} holds for no $r>0$.
Note that $\mu(N)^{-1}$ coincides with the \emph{exponent of convergence} of $N$;
cf.\ \cite[Prop. 2.13, Def. 3.3, Thm. 3.4]{Sanzflatultraholomorphic} and \cite[p.145]{injsurj}.
If $M\in\hyperlink{LCset}{\mathcal{LC}}$ is equivalent to $N$, then $\mu(M)=\mu(N)$.

\subsection{Descendant construction} \label{descendant}
We recall a construction from \cite[Remark 9]{mixedsectorialextensions},
based on \cite[Section 4.1]{whitneyextensionweightmatrix}.
Let $N\in\hyperlink{LCset}{\mathcal{LC}}$ be non-quasianalytic and $r>0$.
The \emph{descendant} of $N^{1/r}$
is the sequence $S = S(N,r)$ defined by $S_p = \sigma_0 \sigma_1 \cdots \sigma_p$, where
$\sigma_0 := 1$ and
$$\sigma_p:=\frac{\tau_1 p}{\tau_p},
\qquad \tau_p:=\frac{p}{(\nu_p)^{1/r}}+\sum_{j\ge p}\Big(\frac{1}{\nu_j}\Big)^{1/r},\qquad p\ge 1. $$
Notice that $S\in\hyperlink{LCset}{\mathcal{LC}}$ is strongly log-convex,
see \cite[Lemma 4.2]{whitneyextensionweightmatrix}; for more of its properties we refer to
\cite[Remark 9]{mixedsectorialextensions}.
For us the sequence $L = L(N,r) \in\hyperlink{LCset}{\mathcal{LC}}$ defined by
\begin{equation}\label{sequenceL}
L:=S^r
\end{equation}
is crucial. We have $L_p = \lambda_0 \lambda_1 \cdots \lambda_p$ with $\lambda :=\sigma^r$.
It has the following properties (see \cite[Remark 9, Lemma 6]{mixedsectorialextensions}):
\begin{itemize}
\item[(i)] 
$(L,N)_{\gamma_r}$ and thus
$\gamma(L,N)\ge r$.

\item[(ii)] $\lambda/\nu$ is bounded
and so $(L,N)_{\gamma_{r'}}$ for all $0<r'\le r$.

\item[(iii)] If $M\in\hyperlink{LCset}{\mathcal{LC}}$ satisfies \ref{gammarmix} and $\mu/\nu$ is bounded,
then also $\mu/\lambda$ is bounded.
Consequently, $L$ is maximal (up to multiplication of $\lambda$ by a constant)
among all sequences $M$ with \ref{gammarmix} and $\mu/\nu$ bounded.

\item[(iv)] $L$ has moderate growth if and only if
\begin{equation}\label{descendantmgequ}
\exists C\ge 1\;\forall k\in\NN_{>0} : \frac{(\nu_{2k})^{1/r}}{(\nu_k)^{1/r}}\le C
+C\frac{(\nu_{2k})^{1/r}}{2k}\sum_{j\ge 2k}\frac{1}{(\nu_j)^{1/r}}.
\end{equation}
\end{itemize}
Moderate growth for $N^{1/r}$ (equivalently for $N$), implies \eqref{descendantmgequ}.
In general, the converse implication is not true; see \cite[Example 1]{mixedsectorialextensions}.

\subsection{Extensions controlled by the order of quasianalyticity}\label{orderofquasim}

Now we are ready to prove a Beurling version of \cite[Theorem 5]{mixedsectorialextensions}.

\begin{theorem}\label{Thm5ultraholombeurling}
Let $N\in\hyperlink{LCset}{\mathcal{LC}}$ satisfy \eqref{beta3}.
Then $\mu(N)>0$.
Let $0<r<\mu(N)$ and suppose that
\eqref{descendantmgequ} holds true for this value $r$.
Then there exists $L\in\hyperlink{LCset}{\mathcal{LC}}$ having moderate growth and with the following properties:
\begin{itemize}
\item[(i)] $\mathcal{B}(\mathcal{A}_{(\widehat{N})}(S_{\gamma}))\supseteq \Lambda_{(\widehat{L})}$ for each $0<\gamma<r$.
\item[(ii)] If $M\in\hyperlink{LCset}{\mathcal{LC}}$ satisfies $M\hyperlink{mvartrian}{\vartriangleleft}L$ and $\omega_M\in\hyperlink{omset1}{\mathcal{W}}$,
then there is an  extension operator
    $$\mathcal E^{M,N}:\Lambda_{\{\widehat{M}\}}\to \mathcal{A}_{(\widehat{N})}(S_{\gamma}).$$
\item[(iii)] $L$ is maximal among all $M\in\hyperlink{LCset}{\mathcal{LC}}$
with \ref{gammarmix} and $\mu/\nu$ bounded.
\end{itemize}
\end{theorem}

\begin{proof}
We have $\mu(N)\ge\gamma(M,N)\ge\gamma(N)$ for any sequence $M\in\hyperlink{LCset}{\mathcal{LC}}$ with
$\mu/\nu$ bounded, see \cite[Lemmas 3, 5, Remark 8]{mixedsectorialextensions},
and $\gamma(N)>0$, by \Cref{beta3remark}. Hence $\mu(N)>0$.

For $r$ as in the assumption, let $L=L(N,r)$ be the sequence defined in \eqref{sequenceL}. Then $L$ has moderate growth,
$\gamma(L,N)\ge r$, $\lambda/\nu$ is bounded, and $L$ satisfies (iii), by the properties listed in \Cref{descendant}.
Furthermore, $S = L^{1/r}$ is strongly log-convex and hence satisfies \eqref{beta3}.
Consequently, by \Cref{beta3remark}(iii),
$\omega_{S}$ satisfies \eqref{om1}.
In view of \eqref{sequenceL} and \eqref{omegaMspower},
also $\omega_{L}$ has the property \eqref{om1}, 
and so $\omega_{L}\in\hyperlink{omset1}{\mathcal{W}}$.
\Cref{beta3remark} also implies that $\omega_{N}\in\hyperlink{omset1}{\mathcal{W}}$.
Thus \Cref{Thm4ultraholombeurling} shows that $L$ satisfies (i) and (ii).
\end{proof}

\providecommand{\bysame}{\leavevmode\hbox to3em{\hrulefill}\thinspace}
\providecommand{\MR}{\relax\ifhmode\unskip\space\fi MR }
\providecommand{\MRhref}[2]{%
  \href{http://www.ams.org/mathscinet-getitem?mr=#1}{#2}
}
\providecommand{\href}[2]{#2}

\end{document}